\documentclass{amsart}

\usepackage{amsmath,amssymb,amsthm}
\usepackage{graphicx,tikz}

\newtheorem*{thm}{Theorem}
\newtheorem*{proposition}{Proposition}

\newtheorem*{lem}{Lemma}

\theoremstyle{definition}

\theoremstyle{remark}

\begin{document}

\title[]{Surrounding the Solution of a Linear \\System of equations from all sides}
\subjclass[2010]{15A09, 15A18, 60D05, 65F10, 90C06} 
\keywords{Linear Systems, Kaczmarz method, Random Geometry}
\thanks{S.S. is supported by the NSF (DMS-1763179) and the Alfred P. Sloan Foundation.}

\author[]{Stefan Steinerberger}
\address{Department of Mathematics, University of Washington, Seattle}
\email{steinerb@uw.edu}

\begin{abstract} Suppose $A \in \mathbb{R}^{n \times n}$ is invertible and we are looking for the solution of $Ax = b$.  Given an initial guess $x_1 \in \mathbb{R}$, 
we show that by reflecting through hyperplanes generated by the rows of $A$, we can generate an infinite sequence $(x_k)_{k=1}^{\infty}$ such that all elements have the same
distance to the solution $x$, i.e. $\|x_k - x\| = \|x_1 - x\|$. If the hyperplanes are chosen at random, averages over the sequence converge and
$$ \mathbb{E} \left\| x - \frac{1}{m} \sum_{k=1}^{m}{ x_k}  \right\| \leq  \frac{1 + \|A\|_F  \|A^{-1}\|}{\sqrt{m}} \cdot\|x-x_1\|.$$
The bound does not depend on the dimension of the matrix.
This introduces a purely geometric way of attacking the problem: are there fast ways of estimating the location of the center of a sphere from knowing many points on the sphere? Our convergence rate (coinciding with that of the Random Kaczmarz method) comes from simple averaging, can one do better? \end{abstract}

\maketitle

\section{Introduction and Results}
\subsection{Introduction.} Let $A \in \mathbb{R}^{n \times n}$ denote an invertible matrix and suppose we are interested in finding the (unique) solution of $Ax = b$. The purpose of this paper is to
point out a simple geometric fact which leads to an iterative method for finding approximations of the solution. We will use $a_i \in \mathbb{R}^n$ to denote the
$i-$th row of the matrix $A$. Interpreting the solution $x$ as the unique intersection of hyperplanes, $\left\langle a_i, x\right\rangle = b_i$ for all $1 \leq i \leq n$, we see that the process of reflecting a point through a hyperplane does not change the distance to the solution $x$ (see Fig. 1).
Therefore, given any point $x_1 \in \mathbb{R}^d$ and any index $1 \leq i \leq n$, 
$$ x_2 = x_1 + 2 \cdot \frac{ b_i - \left\langle x_1, a_i \right\rangle}{\|a_i\|^2} a_i \qquad \mbox{satisfies} \qquad \|x_2 - x\| = \|x_1 - x\|,$$ the distance to the solution remains preserved. 
\begin{figure}[h!]
\begin{center}
\begin{tikzpicture}[scale=0.9]
\draw[thick] (-2, -1) -- (2, 1);
\draw[thick] (0, -1) -- (0, 1);
\draw[thick] (2, -1) -- (-2, 1);
\filldraw (0,0) circle (0.06cm);
\node at (0.4, 0) {$x$};
\filldraw (1.5, 1.2) circle (0.03cm);
\node at (1.25, 1.1) {$x_1$};
\filldraw (1.85, 0.5) circle (0.03cm);
\node at (2.2, 0.4) {$x_2$};
\draw (1.5, 1.2) -- (1.85, 0.5);
\end{tikzpicture}
\end{center}
\caption{Reflecting through a hyperplane, does not change the distance to the solution $x$.}
\end{figure}
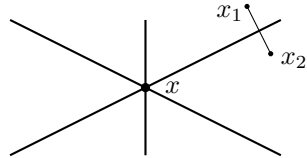
This can be iterated and any choice of sequence of hyperplanes will lead to sequence $(x_k)_{k=1}^{\infty}$ such that the distance to the solution $\|x - x_k\| = \|x - x_1\|$ remains constant along the sequence. We also note that generating the sequence is computationally cheap, the cost of computing $x_{k+1}$ from $x_k$ is merely the cost of computing one inner product and thus $\mathcal{O}(n)$ (since one only requires one row of the matrix, it is also only $\mathcal{O}(n)$ in memory).
\subsection{Random Reflections.}  We have discussed how reflection through hyperplanes given by the rows of the matrix can be used to produce a sequence of points all of which have the same distance to the solution $x$. The arising sequence depends on the initial guess $x_1$ and the sequence in which the hyperplanes are chosen. We propose to pick the hyperplanes randomly: more precisely, we suggest to pick the $i-$th hyperplane with likelihood proportional to $\|a_i\|^2$. As discussed above, this gives rise to an infinite sequence $(x_k)_{k=1}^{\infty}$ such that all elements are at the same distance from the solution, i.e. 
$$\|x_k - x\| = \|x_1 - x\|.$$
We can estimate the convergence rate of their averages.

\begin{thm} If the $i-$th hyperplane is picked with likelihood proportional to $\|a_i\|^2$, the arising random sequence of points $(x_k)_{k=1}^{\infty}$ satisfies
$$ \mathbb{E} \left\| x - \frac{1}{m} \sum_{k=1}^{m}{ x_k}  \right\| \leq  \frac{1 + \|A\|_F  \|A^{-1}\|}{\sqrt{m}} \cdot\|x-x_1\|,$$
where $\|A\|_F$ is the Frobenius norm and $\|A^{-1}\|$ the operator norm of the inverse.
\end{thm}
\textbf{Remarks.}
\begin{itemize}
\item Naturally, one would run the algorithm for a while, replace $x_m$ by the average of $x_1, \dots, x_m$ and then start the
algorithm again with this average as new initial point. This always decreases the distance to the solution.
\item  The Theorem implies that in order to have the average be a factor 2 closer to the solution than the initial guess $\|x -x_1\|$ requires $\sim \|A^{-1}\|^2 \|A\|_F^2$ steps. This
is equivalent to the performance of the Random Kaczmarz method (see \S 1.4) up to constants.
\item It is an interesting question whether the rate (or the dependency on the matrix) can be improved by replacing the average of $x_1, \dots, x_m$ by another method that produces a better guess for the center of the sphere, see \S 1.3.
\item It is an interesting question whether the rate (or the dependency on the matrix) can be improved by picking the hyperplanes deterministically in a clever (and computationally cheap) way, see also \S 3.
\end{itemize}

\subsection{A Geometry Problem.} 
What we believe to be particularly interesting is that this approach leads to a new way of attacking the underlying problem: we can now (cheaply) convert any linear system to the problem of finding the center of a sphere from points on the sphere.
\begin{quote}
\textbf{Problem.} Given at least $n+1$ points on a sphere in $\mathbb{R}^{n}$, how would one quickly determine an approximation of the center of the sphere? Does it help if one has $c\cdot n$ points?
\end{quote}
This seems like a very basic mathematical problem; surely there should be many different ways of attacking it!  We have little control over where the $c \cdot n$ points lie, some of them might be localized in small regions of the sphere which certainly makes the problem harder (see \S 3)-- then again, the current goal would be to improve on the average of the points as an estimate for the center of the sphere. 
It will not come as a surprise that there is a classical way of solving the problem: given $x_1, \dots, x_{n+1}$ points on a sphere in $\mathbb{R}^{n}$, we note that the vector $r$ with the property that $x_1 + r$ is the center of the sphere satisfies
$$\forall~2 \leq i \leq n+1 \qquad \quad \left\langle x_i - x_1 , 2r \right\rangle = \|x_i - x_1\|^2$$
This is the classical Theorem of Thales in disguise (see Fig. 2). This leaves us with a linear system of equations to determine $r$. As such, it seems that we did not gain very much
since we exchanged the solution of one linear system by another: however, it is conceivable that the new linear system is better conditioned and can be solved quicker. If that were
the case, we would have gained a fair amount. Moreover, we are not limited to using $n+1$ points on the sphere, we can actually choose more: this would then allow us to replace
an $n \times n$ matrix by an overdetermined $cn \times n$ system which may lead to additional stability.

\begin{figure}[h!]
\begin{center}
\begin{tikzpicture}[scale=1.5]
   \draw [black,thick,domain=0:180] plot ({2*cos(\x)}, {2*sin(\x)});
   \draw [thick] (-2,0) -- (2,0);
   \filldraw (2,0) circle (0.06cm);
   \node at (2.2, -0.3) {$x_1$};
      \filldraw (-0.75,1.85) circle (0.06cm);
      \node at (-1, 2.1) {$x_i$};
   \node at (2.2, -0.3) {$x_1$};
   \draw[->, ultra thick] (2,0) -- (-2,0);
   \node at (-2, -0.2) {$2r$};
   \draw[->,  thick] (2,0) -- (-0.7, 1.77);
\end{tikzpicture}
\end{center}
\caption{Thales' Theorem guarantees $\left\langle x_i - x_1, 2r \right\rangle = \|x_i - x_1\|^2$.}
\end{figure}
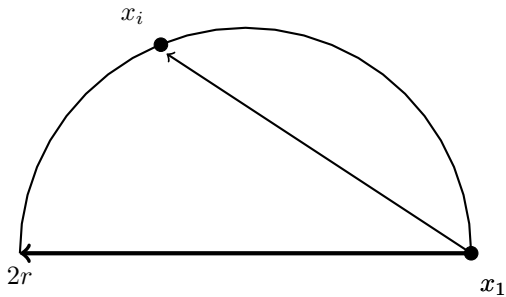 
\textbf{An Example.} We illustrate this with a numerical example. We generate 1000 random matrices $A \in \mathbb{R}^{50 \times 50}$ by picking each entry to be $\sim \mathcal{N}(0,1)-$distributed (in an i.i.d. fashion) and then normalize the rows to have norm 1. For each such matrix, we generate $B \in \mathbb{R}^{200 \times 50}$ by starting in a fixed point and performing 5000 random reflections through hyperplanes: we take every 25-th element of the sequence (for the purpose of decorrelation, consecutive elements have stronger correlation) and use the arising 200 points to generate the matrix $B$ that arises from trying to find the center of the sphere as outlined above. We see that
$$ \mathbb{E} ~\frac{1}{\|A^{-1}\| \cdot \|A\|_F} \sim 0.0019 \quad \mbox{and} \quad  \mathbb{E}~ \frac{1}{\|B^{-1}\| \cdot \|B\|_F} \sim 0.0045.$$
The condition number has a similar improvement.
This shows that, at least with respect to our method or the Random Kaczmarz method, solving the new linear system is more than twice as fast as solving the original one: naturally, there is the pre-computation cost of computing the 5000 elements of the random sequence and assembling the new matrix. Ultimately, whether switching to this different linear system is reasonable depends on how accurate of an approximation one wants. \\

We conclude by emphasizing once more that there should be other ways by which one can cheaply approximate the center of a sphere from sufficiently many points on the sphere: any method that improves on the simple idea of averaging would lead, by the mechanism above, to an iterative method for solving linear system of equations that does at least as well as the Random Kaczmarz method.

\subsection{Related methods}  We quickly discuss two related methods: first and foremost, there is a great degree of similarity to the Random Kaczmarz method where, instead of projecting through the hyperplane, one projects $x_k$ onto a randomly chosen hyperplane to obtain a better approximation $x_{k+1}$ of the solution. This is also known as the \textit{Algebraic Reconstruction Technique} (ART) in computer tomography \cite{gordon, her, her2, natterer}, the \textit{Projection onto Convex Sets Method} \cite{cenker, deutsch, deutsch2, gal, sezan} and the \textit{Randomized Kaczmarz method} \cite{eldar, elf, gordon2, gower, gower2, lee, leventhal, liu, ma, moor, need, need2, need25, need3, need4, nutini, popa, stein, stein2, stein3, strohmer, zhang, zouz}.
 Strohmer \& Vershynin \cite{strohmer} showed that 
$$ \mathbb{E} \left\| x_k - x \right\|_2^2 \leq \left(1 - \frac{1}{\|A\|_F^2 \| A^{-1}\|_2^2}\right)^k \|x_0 - x\|_2^2,$$
where $\|A^{-1}\|_2$ is the operator norm of the inverse and $\|A\|_F$ is the Frobenius norm.
 In order to improve the distance to the solution by a factor of $2$, we require
$$ k \sim \|A\|_F^2 \cdot \|A^{-1}\|^2$$
iteration steps which is the same as in our method (up to constants).
As a second related method, we emphasize the Cimmino method \cite{cim} which sets
$$ x_{k+1} = x_k + \frac{\lambda}{M} \sum_{i=1}^{n} \frac{b_i - \left\langle x_k, a_i\right\rangle}{\|a_i\|^2} a_i,$$
where $\lambda$ is a real parameter: for $\lambda = 2$, we end up projecting through each hyperplane in the same distance-preserving manner and then end up averaging. Cimmino's method is as costly as $n$ steps of Randomized Kaczmarz or our method. 
These two methods form the basis of certain types of iterative procedures that are frequently employed in computed tomography \cite{jiang}. Other such methods are the Simultaneous Algebraic Reconstruction Technique  (SART) \cite{sart}, Component Averaging (CAV) \cite{cav} and Diagonally Relaxed Orthogonal Projections (DROP) \cite{drop}, we also refer to 
\cite{elf2, hansen, her, natterer}.

\section{Proof of the Theorem}
A crucial ingredient is the random linear operator $R:\mathbb{R}^n \rightarrow \mathbb{R}^n$ given by
$$ Rx = x + 2 \cdot \frac{ b_i - \left\langle x, a_i \right\rangle}{\|a_i\|^2} a_i \qquad \mbox{with likelihood}~~\frac{\|a_i\|_2}{\|A\|_F^2}.$$
We emphasize that $R$ will always denote a random operator and we will use $R^k = R_1 R_2 \dots R_k$ to denote the application of $k$ independent copies of $R$. Making the ansatz $x_k = x + y_k$, we observe that
\begin{align*}
 x + y_{k+1} &=  x_{k+1} = x_k + 2 \cdot \frac{ b_i - \left\langle x_k, a_i \right\rangle}{\|a_i\|^2} a_i \\
 &= x + y_k + 2 \cdot \frac{ b_i - \left\langle x + y_k, a_i \right\rangle}{\|a_i\|^2} a_i\\
 &= x + y_k - 2 \cdot \frac{ \left\langle  y_k, a_i \right\rangle}{\|a_i\|^2} a_i.
\end{align*}
Subtracting $x$ from both sides leads to the iteration
$$ y_{k+1} = y_k - 2 \cdot \frac{ \left\langle  y_k, a_i \right\rangle}{\|a_i\|^2} a_i$$
which corresponds to the reflection operator applied to the problem $Ax = 0$. This 
 shows that the problem is invariant under translation (hardly surprising from the way it has been constructed). It thus suffices to study the random operator when applied to the problem $Ax = 0$ which we again denote by $R$, i.e.
$$ R x =  x - 2 \cdot \frac{\left\langle x, a_i \right\rangle}{\|a_i\|^2} a_i \qquad \mbox{with likelihood}~~\frac{\|a_i\|_2}{\|A\|_F^2}.$$
Note that $\|Rx\| = \|x\|$ since
$$ \left\langle x -  \frac{\left\langle x, a_i \right\rangle}{\|a_i\|^2} a_i, a_i \right\rangle =0$$
and thus, for all $\lambda \in \mathbb{R}$ and, in particular, $\lambda = \left\langle x, a_i \right\rangle/\|a_i\|^2$,
$$  \left\| \left(x -  \frac{\left\langle x, a_i \right\rangle}{\|a_i\|^2} a_i \right) +  \lambda a_i \right\| = \left\| \left(x -  \frac{\left\langle x, a_i \right\rangle}{\|a_i\|^2} a_i \right) -  \lambda a_i \right\|.$$
We study the operator $R$ throughout the rest of the proof. The main ingredient is to
show that $x$ and $R^k x$ are decorrelated: in expectation, their inner product is small. We expect this effect to get stronger as $k$ increases: the more random reflection we have performed, the less memory of the original point should remain.
\begin{lem} We have, for any $x \in \mathbb{R}^n$, and any $k \in \mathbb{N}$,
$$ \left| \mathbb{E}  \left\langle x, R^k x\right\rangle \right| \leq \left(1 - \frac{2}{\|A^{-1}\|^2 \|A\|_F^2}\right)^k \|x\|^2.$$
\end{lem}
\begin{proof}
We start with the case $k=1$. There, we have
\begin{align*}
\left|  \mathbb{E}  \left\langle x, Rx \right\rangle \right| &= \left| \sum_{k=1}^{n}{\frac{\|a_k\|^2}{\|A\|_F^2} \left\langle x, x - 2\frac{\left\langle x, a_k \right\rangle }{\|a_k\|^2} a_k \right\rangle } \right| =\left| \|x\|^2  - \frac{2}{\|A\|_F^2} \sum_{k=1}^{n} \left\langle x, a_k\right\rangle^2 \right|\\
&=\left| \|x\|^2 - \frac{2}{\|A\|_F^2} \|Ax_k\|^2 \right| = \left| \left\langle x, x\right\rangle - \frac{2}{\|A\|_F^2} \left\langle Ax, Ax \right\rangle \right| \\
&= \left| \left\langle x - \frac{2}{\|A\|_F^2} A^T A x, x \right\rangle \right| \leq \left\| x - \frac{2}{\|A\|_F^2} A^T A x \right\| \left\|x \right\|.
\end{align*}
We can simplify this further to
$$ \left\| x - \frac{2}{\|A\|_F^2} A^T A x \right\| \|x\| \leq \left\| \mbox{Id}_{n \times n} - \frac{2}{\|A\|_F^2} A^T A \right\| \|x\|^2.$$
It remains to bound the norm of this symmetric matrix. We have
$$\sigma_n^2 \|x\| \leq  \|A^T A x \| \leq \sigma_1^2 \|x\|.$$
We will use this to conclude $$ \left\| \mbox{Id}_{n \times n} - \frac{2}{\|A\|_F^2} A^T A \right\|  \leq \max\left\{ 1 - \frac{2 \sigma_n^2}{\|A\|_F^2} , \frac{2\sigma_1^2}{\|A\|_F^2} - 1 \right\}$$
as follows: using $v_i$ to denote the singular vector associated to $A$, then these are the eigenvectors of $A^T A$ corresponding to eigenvalue $\sigma_i^2$. Thus, we have
\begin{align*}
 \left(\mbox{Id}_{n \times n} - \frac{2}{\|A\|_F^2} A^T A\right)x &= \sum_{i=1}^{n} \left\langle x, v_i \right\rangle v_i - \frac{2}{\|A\|_F^2} \sum_{i=1}^{n} \sigma_i^2 \left\langle x, v_i \right\rangle v_i  \\
 &=  \sum_{i=1}^{n} \left(1 - \frac{2}{\|A\|_F^2}  \sigma_i^2 \right) \left\langle x, v_i \right\rangle v_i.
 \end{align*}
Observing that $0< \sigma_n \leq \dots \leq \sigma_1$ and hence, for all $1 \leq i \leq n$,
$$ -1 \leq 1 - \frac{2}{\|A\|_F^2}  \sigma_1^2 \leq 1 - \frac{2}{\|A\|_F^2}  \sigma_i^2 \leq 1 - \frac{2}{\|A\|_F^2}  \sigma_n^2 \leq 1$$
from which we deduce
$$ \left|  1 - \frac{2}{\|A\|_F^2}  \sigma_i^2 \right| \leq \max\left\{ 1 - \frac{2 \sigma_n^2}{\|A\|_F^2} , \frac{2\sigma_1^2}{\|A\|_F^2} - 1 \right\}.$$
We recall that the Frobenius norm is related to the sum of the squared singular values via
$$ \|A\|_F^2 = \sum_{k=1}^{n}{\sigma_k^2}$$
 and thus
\begin{align*}
\frac{2\sigma_1^2}{\|A\|_F^2} - 1 &=  \frac{2\sigma_1^2 - \sum_{k=1}^{n}\sigma_k^2}{\|A\|_F^2} =  \frac{\sigma_1^2 - \sum_{k=2}^{n}\sigma_k^2}{\|A\|_F^2} \\
&\leq \frac{\sigma_1^2 + \sigma_n^2 - 2\sigma_n^2}{\|A\|_F^2} \leq 1 - \frac{2 \sigma_n^2}{\|A\|_F^2}.
\end{align*}
Therefore, recalling that for invertible $A \in \mathbb{R}^{n \times n}$, we have $\sigma_n = \|A^{-1}\|$
$$ \left\| \mbox{Id}_{n \times n} - \frac{2}{\|A\|_F^2} A^T A \right\|  \leq  1 - \frac{2 \sigma_n^2}{\|A\|_F^2} = 1 - \frac{2}{\|A\|_F^2 \|A^{-1}\|^2}$$
and the case $k=1$ follows. Let now $k \geq 2$ and let $x,y \in \mathbb{R}^n$ be arbitrary. Then
\begin{align*}
  \mathbb{E} \left\langle y, R^k x \right\rangle  &= \sum_{k=1}^{n}{\frac{\|a_k\|^2}{\|A\|_F^2} \left\langle y, R^{k-1} x - 2\frac{\left\langle R^{k-1} x, a_k \right\rangle }{\|a_k\|^2} a_k \right\rangle } \\
&=   \left\langle y, R^{k-1} x\right\rangle - \frac{2}{\|A\|_F^2} \sum_{k=1}^{n} \left\langle y, a_k\right\rangle \left\langle R^{k-1} x, a_k \right\rangle \\
&=   \left\langle y, R^{k-1} x\right\rangle - \frac{2}{\|A\|_F^2} \left\langle Ay, A R^{k-1} x\right\rangle\\
&=   \left\langle y - \frac{2}{\|A\|_F^2} A^T A y, R^{k-1} x \right\rangle  \\
&=   \left\langle \left(\mbox{Id}_{n \times n}  - \frac{2}{\|A\|_F^2} A^T A\right) y, R^{k-1} x \right\rangle.
\end{align*}
Since $y$ was completely arbitrary, we can use the identity iteratively to see that
$$   \mathbb{E} \left\langle y, R^k x \right\rangle   =  \left\langle \left(\mbox{Id}_{n \times n}  - \frac{2}{\|A\|_F^2} A^T A\right)^k y,  x \right\rangle$$
from which we deduce
$$ \left|  \mathbb{E} \left\langle y, R^k x \right\rangle \right| \leq \left( 1 - \frac{2}{\|A\|_F^2 \|A^{-1}\|^2} \right)^k \left| \left\langle x, y \right\rangle \right|.$$
\end{proof}

\begin{proof}[Proof of Theorem 1] We square the expected value and use the fact that long-range interactions are exponentially decaying in expectation. Clearly,
\begin{align*}
 \mathbb{E} \left\| \sum_{k=1}^{m}{ R^{k-1} x} \right\|^2 &= \mathbb{E} \left\langle \sum_{k=1}^{m}{ R^{k-1} x}, \sum_{k=1}^{m}{ R^{k-1} x} \right\rangle \\
 &=\mathbb{E} \sum_{k, \ell=1}^{m}  \left\langle R^{k-1} x, R^{\ell-1} x\right\rangle.
 \end{align*}
 We first observe that $R$ preserves the norm and thus we can separate diagonal and off-diagonal terms into 
 \begin{align*}
  \mathbb{E} \sum_{k, \ell=1}^{m}  \left\langle R^{k-1} x, R^{\ell-1} x\right\rangle &= m\cdot \|x\|^2 + \mathbb{E} \sum_{k, \ell=1 \atop k \neq \ell}^{m} \left\langle R^{k-1} x, R^{\ell-1} x\right\rangle\\
 &= m\cdot \|x\|^2 + 2 \cdot \mathbb{E}\sum_{k=1}^{m-1} \sum_{\ell=k+1}^{m}  \left\langle R^{k-1} x, R^{\ell-1} x\right\rangle.
 \end{align*}
 We now fix $k$ and estimate, using the Lemma and $\|Rx\| = \|x\|$, 
\begin{align*}
 \left| \mathbb{E} \sum_{\ell=k+1}^{m}  \left\langle R^{k-1} x, R^{\ell-1} x\right\rangle \right| &\leq \sum_{\ell=k+1}^{m}  \left| \mathbb{E} \left\langle \left[R^{k-1} x\right], R^{\ell-k} \left[ R^{k-1} x \right]\right\rangle \right| \\
 &\leq  \sum_{\ell=k+1}^{m} \left( 1 - \frac{2 }{\|A\|_F^2 \|A^{-1}\|^2}\right)^{\ell-k} \|x\|^2.
 \end{align*}
 Abbreviating $ q = 1 - 2/(\|A\|_F^2 \|A^{-1}\|^2)$
 we have
 \begin{align*}
   \mathbb{E} \left\| \sum_{k=1}^{m}{ R^{k-1} x} \right\|^2 &\leq m \cdot \|x\|^2 + 2 \|x\|^2 \sum_{k=1}^{m-1} \sum_{\ell=1}^{m-k} q^{\ell} \\
   &\leq m \cdot \|x\|^2 + 2 \|x\|^2  \frac{m}{1-q}\\
   &= m\cdot \|x\|^2 \cdot \left( 1+    \|A\|_F^2 \|A^{-1}\|^2\right)
 \end{align*}  
from which the result follows after dividing by $m^2$ on both sides.
\end{proof}

\section{Concluding Remarks}
We conclude with some remarks. As above, by linearity, we can restrict ourselves to solving $Ax = 0$ via random reflections that act as
$$ R x =  x - 2 \cdot \frac{\left\langle x, a_i \right\rangle}{\|a_i\|^2} a_i \qquad \mbox{with likelihood}~~\frac{\|a_i\|_2}{\|A\|_F^2}.$$
We can also assume, again by linearity, that $\|x_1\|=1$ in which case the arising sequence lies on the unit sphere $\left\{x \in \mathbb{R}^n: \|x\|=1\right\}$.
It would be interesting to understand whether, as the number of points tends to infinity, whether they end up being distributed according to some limiting distribution and, if that is the case, how quickly convergence to
the limiting distribution happens. It seems likely that both limiting distribution
and convergence rate should depend on the singular vectors and singular
values. A supporting fact (motivated by \cite[Theorem 1]{stein}) is as follows.
\begin{proposition} Let $v_{\ell}$ be a singular vector of $A$. Then, for any $x \in \mathbb{R}^n$,
$$\mathbb{E} \left\langle R^kx, v_{\ell} \right\rangle = \left(1 - \frac{2 \sigma_{\ell}^2}{\|A\|_F^2} \right)^k \left\langle x, v_{\ell} \right\rangle.$$
\end{proposition}
\begin{proof} A simple computation shows
\begin{align*}
\mathbb{E} \left\langle Rx, v_{\ell} \right\rangle &= \sum_{k=1}^{n} \frac{\|a_k\|^2}{\|A\|_F^2} \left\langle x - 2\frac{\left\langle x, a_k\right\rangle}{\|a_k\|^2} a_k, v_{\ell} \right\rangle \\
&= \left\langle x, v_{\ell} \right\rangle - \frac{2}{\|A\|_F^2} \sum_{k=1}^{n} \left\langle x, a_k\right\rangle  \left\langle a_k , v_{\ell} \right\rangle \\
&= \left\langle x, v_{\ell} \right\rangle - \frac{2}{\|A\|_F^2} \left\langle Ax, A v_{\ell} \right\rangle \\
&= \left\langle x, v_{\ell} \right\rangle - \frac{2}{\|A\|_F^2} \left\langle x, A^TA v_{\ell} \right\rangle \\
&= \left(1 - \frac{2 \sigma_{\ell}^2}{\|A\|_F^2} \right) \left\langle x, v_{\ell} \right\rangle.
\end{align*}
Naturally, since multiple applications of $R$ are independent, this also shows that
$$\mathbb{E} \left\langle R^kx, v_{\ell} \right\rangle = \left(1 - \frac{2 \sigma_{\ell}^2}{\|A\|_F^2} \right)^k \left\langle x, v_{\ell} \right\rangle.$$
\end{proof}

This suggests that small singular values are connected to bad mixing properties (and vice versa); we note a similar issue arises for the Random
Kaczmarz method, see \cite{stein, stein3}.  In summary, we see that\\

\begin{center}
\begin{tikzpicture}
\node at (0,0) {Solving Linear Systems};
\node at (0,-0.4) {is hard};
\node at (2.2,-0.2) {$\Leftrightarrow$};
\node at (4.3,-0.2) {Small Singular Values};
\node at (6.4,-0.2) {$\Leftrightarrow$};
\node at (7.5,-0) {$R$ mixes};
\node at (7.5,-0.4) {slowly};
\end{tikzpicture}
\end{center}
 We illustrate this with yet another identity (related to \cite[Theorem 3]{stein}).
 
\begin{proposition}
If $x$ is the linear combination of singular vectors corresponding to small singular values, then $Rx$ is often being projected to a point rather close to $x$.
\end{proposition} 
 \begin{proof}  Recalling an argument from the proof of the Lemma,
 \begin{align*}
 \mathbb{E}  \left\langle Rx, x \right\rangle  &= \sum_{k=1}^{n} \frac{\|a_k\|^2}{\|A\|_F^2} \left\langle x - 2\frac{\left\langle x, a_k\right\rangle}{\|a_k\|^2} a_k, x \right\rangle\\
&=  \|x\|^2 - \frac{2}{\|A\|_F^2} \sum_{k=1}^{n} \left\langle x, \left\langle x, a_k\right\rangle a_k \right\rangle\\
&=  \|x\|^2 - \frac{2}{\|A\|_F^2} \sum_{k=1}^{n} \left\langle x,  a_k \right\rangle^2\\
&= \|x\|^2 - \frac{2}{\|A\|_F^2}  \|Ax\|^2.
\end{align*}
This implies the desired statement: for vectors that are linear combinations of small singular vectors, we have that $\|Ax\|^2$ is quite small implying that $ \mathbb{E}  \left\langle Rx, x \right\rangle$ is close to 1. Since $\left\langle Rx, x \right\rangle \leq \|Rx\| \|x\| = 1$, this means, via Markov's inequality, that $Rx$ must, typically, be very close to $x$. 
\end{proof}
 
 One of the main points of this paper is that solving linear systems is at most as hard as finding the center of sphere from knowing many points on it.  Badly conditioned linear systems are hard to solve because singular vectors are small which corresponds to random reflections $R$ slowly exploring the sphere. 
 \begin{quote}
 \textbf{Question.} Are there computationally cheap ways of selecting the order of reflection deterministically in such a way that points end up being better distributed on the sphere?
 \end{quote}

\end{document}